\documentclass[a4paper,12pt,draft]{amsart}
\usepackage[margin=30mm]{geometry}
\usepackage{amssymb, amsmath, amsthm, amscd}

\usepackage[T1]{fontenc}
\usepackage{eucal,mathrsfs,dsfont}
\usepackage{color}
\usepackage{verbatim}
\usepackage{enumerate}

\newcommand{\rene}{\color{black}}

\newcommand{\normal}{\color{black}}

\renewcommand{\leq}{\leqslant}
\renewcommand{\geq}{\geqslant}
\renewcommand{\le}{\leqslant}
\renewcommand{\ge}{\geqslant}
\renewcommand{\Re}{\ensuremath{\mathop{\mathrm{Re}}}}

\newcommand{\scalar}[1]{\langle#1\rangle}
\newcommand{\nnorm}[1]{\|#1\|}
\definecolor{mno}{rgb}{0.5,0.1,0.5}

\newcommand{\R}{\mathds R}
\newcommand{\real}{\mathds R}
\newcommand{\Rd}{{\mathds R^d}}

\newcommand{\Ee}{\mathds E}

\newcommand{\I}{\mathds 1}

\newcommand{\Bb}{\mathscr{B}}

\newtheorem{theorem}{Theorem}
\newtheorem{lemma}[theorem]{Lemma}
\newtheorem{proposition}[theorem]{Proposition}

\theoremstyle{definition}

\newtheorem{remark}[theorem]{Remark}

\newtheorem*{exm}{Example}
\newtheorem*{ack}{Acknowledgement}

\begin{document}

\title[Functional Inequalities and Subordination]{\bfseries Functional Inequalities and Subordination:\\
Stability of Nash and Poincar\'{e} inequalities}
\author{Ren\'{e} L.\ Schilling \and  Jian Wang}
\thanks{\emph{R.L.\ Schilling:} TU Dresden, Institut f\"{u}r Mathematische Stochastik, 01062 Dresden, Germany.
    \texttt{rene.schilling@tu-dresden.de}
    }
\thanks{\emph{J.\ Wang 
:}
   School of Mathematics and Computer Science, Fujian Normal University, 350007 Fuzhou, P.R. China \emph{and} TU Dresden, Institut f\"{u}r Mathematische Stochastik, 01062 Dresden, Germany.
   \texttt{jianwang@fjnu.edu.cn
   }
    }

\maketitle

\begin{abstract}
We show that certain functional inequalities, e.g.\ Nash-type and
Poincar\'e-type inequalities, for infinitesimal generators of $C_0$
semigroups are preserved under subordination in the sense of
Bochner. Our result improves \cite[Theorem 1.3]{BM} by A.\ Bendikov
and P.\ Maheux for fractional powers, and it also holds for
non-symmetric settings. As an application, we will derive
hypercontractivity, supercontractivity and ultracontractivity of
subordinate semigroups.

\medskip


\medskip

\end{abstract}

\section{Introduction}\label{section1}
In this note we show that certain functional inequalities are preserved under subordination in the sense of Bochner.

Bochner's subordination is a method to get new semigroups from a given one.
Let us briefly summarize the main facts about subordination; our main reference is the monograph \cite{RSZ},
in particular Chapter 12. Let $(T_t)_{t\ge0}$ be a strongly continuous ($C_0$) contraction semigroup on a Banach space $(\Bb,\|\cdot\|)$. The infinitesimal generator is the operator
\begin{align*}
    Au  &:= \lim_{t\to 0} \frac{u - T_tu}{t},\\
    D(A) &:= \left\{u\in\Bb \::\: \lim_{t\to 0} \frac{u - T_tu}{t}\text{\ \ exists in the strong sense}\right\}.
\end{align*}

A \emph{subordinator} is a vaguely continuous convolution semigroup of sub-probability measures $(\mu_t)_{t\ge 0}$ on $[0,\infty)$. Subordinators are uniquely characterized by the Laplace transform:
$$
    \mathscr L\mu_t(\lambda) = \int_{[0,\infty)} e^{-s\lambda}\,\mu_t(ds) = e^{-tf(\lambda)}
    \quad\text{for all\ \ } t\ge0
    \text{\ \ and \ \ } \lambda \ge 0.
$$
The characteristic exponent $f\!:\!(0,\infty)\to(0,\infty)$ is a \emph{Bernstein function}, i.e.\ a function of the form
\begin{equation}\label{bfunction}
    f(\lambda)=a+b\lambda+\int_{(0,\infty)}\big(1-e^{-t\lambda}\big)\,\nu(dt),
\end{equation}
where $a, b\ge 0$ are nonnegative constants and $\nu$ is a nonnegative measure on $(0,\infty)$ satisfying $\int_{(0,\infty)}\big(1\wedge t\big)\,\nu(dt)<\infty$. There are one-to-one relations between the triplet $(a,b,\nu)$, the Bernstein function $f$ and the subordinator $(\mu_t)_{t\geq 0}$. Among the most prominent examples of Bernstein functions are the fractional powers $f_\alpha(\lambda) = \lambda^\alpha$, $0<\alpha\le1$. The Bochner integral
$$
    T_t^fu := \int_{[0,\infty)}T_su\,\mu_t(ds),\quad t\ge 0,\; u\in\Bb,
$$
defines a strongly continuous contraction semigroup on $\Bb$. We call $(T_t^f)_{t\ge0}$ \emph{subordinate} to $(T_t)_{t\geq 0}$ (with respect to the subordinator $(\mu_t)_{t\geq 0}$ or the Bernstein function $f$). Subordination preserves many additional properties of the original semigroup. For example, on a Hilbert space, $(T_t^f)_{t\geq 0}$ inherits symmetry from $(T_t)_{t\geq 0}$ and on an ordered Banach space $(T_t^f)_{t\geq 0}$ is sub-Markovian whenever $(T_t)_{t\geq 0}$ is. Let us write $(A^f,D(A^f))$ for the generator of $(T_t^f)_{t\ge0}$; it is known that $D(A)$ is an operator core of $A^f$ and that $A^f$ is given by Phillip's formula
\begin{equation}\label{e-phillips}
    A^f u = au + b\,Au + \int_{(0,\infty)} (u-T_su)\,\nu(ds),\quad u\in D(A).
\end{equation}
Here $(a,b,\nu)$ is the defining triplet for $f$ as in \eqref{bfunction}.

Bochner's subordination gives rise to a functional calculus for generators of $C_0$ contraction semigroups. In many situations this functional calculus coincides with classical functional calculi, e.g.\ the spectral calculus in Hilbert space or the Dunford-Taylor spectral calculus in Banach space, cf.\ \cite{ber-boy-del} and \cite{RSZ}. It is, therefore, natural to write $f(A)$ instead of $A^f$.

\bigskip
From now on we will use $\Bb = L^2(X,m)$ where $(X,m)$ is a measure space with a $\sigma$-finite measure $m$. We write $\scalar{\cdot,\cdot}$ and $\|\cdot\|_2$ for the scalar product and norm in $L^2$, respectively; $\|\cdot\|_1$ denotes the norm in $L^1(X,m)$. To compare our result with \cite[Theorem 1.3]{BM}, we start with Nash-type inequalities. Our main contribution to this type of functional inequalities are the following two results.

\begin{theorem}\label{th1} \textup{\bfseries (symmetric case)}
    Let $(T_t)_{t\geq 0}$ be a strongly continuous contraction semigroup of symmetric operators on $L^2(X,m)$ and assume that for each $t\ge0$, $T_t|_{L^2(X,m)\cap L^1(X,m)}$ has an extension which is a contraction on $L^1(X,m)$, i.e.\ $\|T_tu\|_1\le \|u\|_1$ for all $u\in L^1(X,m)\cap L^2(X,m)$. Suppose that the generator $(A,D(A))$ satisfies the following Nash-type inequality:
    \begin{equation}\label{th11}
        \|u\|_2^2\,B\left( \|u\|_2^2\right)
        \le \langle A\,u,u\rangle,
        \quad u\in D(A),\; \|u\|_1=1,
    \end{equation}
    where $B:(0,\infty)\rightarrow(0,\infty)$ is any increasing function. Then, for any Bernstein function $f$, the generator $f(A)$ of the subordinate semigroup satisfies
    \begin{equation}\label{th12}
        \frac{\|u\|_2^2}{2}\,f\left(B\left(\frac{\|u\|_2^2}{2}\right)\right)
        \le \langle f(A)\,u,u\rangle,
        \quad u\in D(f(A)),\; \|u\|_1=1.
    \end{equation}
\end{theorem}

\begin{remark} For fractional powers $A^\alpha$, $0<\alpha<1$, the result of Theorem \ref{th1} is due to Bendikov and Maheux \cite[Theorem 1.3]{BM}; this corresponds to the Bernstein functions $f(\lambda)=\lambda^\alpha$. Our result is valid for \emph{all} Bernstein functions, hence, for \emph{all subordinate generators} $f(A)$.
Note that \cite[Theorem 1.3]{BM} claims that
$$
        c_1{\|u\|_2^2}\,\left(B\left(c_2{\|u\|_2^2}\right)\right)^\alpha
        \le \langle A^\alpha\,u,u\rangle,
        \quad u\in D(A^\alpha),\; \|u\|_1=1,
$$
holds for all $0<\alpha <1$ with $c_1=c_2=1$, but a close inspection of the proof in \cite{BM} reveals that one has to assume, in general, $c_1, c_2 \in (0,1)$. Note that Theorem \ref{th1} yields $c_1=c_2=1/2.$
\end{remark}

If $(T_t)_{t\geq 0}$ is not symmetric, we still have the following result.
\begin{theorem}\label{th2} \textup{\textbf{(non-symmetric case)}}
     Let $(T_t)_{t\geq 0}$ be a strongly continuous \rene contraction semigroup on \normal $L^2(X,m)$ and assume that for each $t\ge0$, $T_t|_{L^2(X,m)\cap L^1(X,m)}$ has an extension which is a contraction on $L^1(X,m)$. Suppose that the generator $(A,D(A))$ satisfies the following Nash-type inequality:
    \begin{equation}\label{th21}
        \|u\|_2^2\,B\left( \|u\|_2^2\right)
        \le \Re \langle A\,u,u\rangle,
        \quad u\in D(A),\; \|u\|_1=1,
    \end{equation}
    where $B:(0,\infty)\rightarrow(0,\infty)$ is any increasing function. Then, for any Bernstein function $f$, the generator $f(A)$ of the subordinate semigroup satisfies
    \begin{equation}\label{th22}
        \frac{\|u\|_2^2}{4}\,f\left(2B\left(\frac{\|u\|_2^2}{2}\right)\right)
        \le \Re \langle f(A)\,u,u\rangle,
        \quad u\in D(f(A)),\; \|u\|_1=1.
    \end{equation}
\end{theorem}

\begin{remark} (i) The assumption that $T_t$ is a contraction both in $L^2(X,m)$ and $L^1(X,m)$ is often satisfied in concrete situations. Assume that $(T_t)_{t\geq 0}$ is a strongly continuous contraction semigroup on $L^2(X,m)$ such that the operators $T_t$ are symmetric and sub-Markovian---i.e.\ $0\leq T_tv\leq 1$\,\;a.e. for all $0\leq v\leq 1$\,\;$m$-a.e. Then the following argument shows that $T_t|_{L^2(X,m)\cap L^1(X,m)}$ is a contraction on $L^1(X,m)$:
$$
    \scalar{T_t u, v}
    = \scalar{u,T_t v}
    \leq \scalar{|u|,\nnorm v_\infty}
    = \nnorm v_\infty \nnorm{u}_1
    \quad u\in L^2\cap L^1,\; v\in L^1\cap L^\infty.
$$
In general, a sub-Markovian $L^2$-contraction operator $T_t$ is also an $L^1$-contraction if, and only if, the $L^2$-adjoint $T_t^*$ is a sub-Markovian operator, cf.\  \cite[Lemma 2]{SCH}.

\medskip\noindent
(ii) From \eqref{bfunction} it follows that Bernstein functions are subadditive, thus
$$
    \frac 12\,f(2x) \leq f(x).
$$
This shows that, for symmetric semigroups, \eqref{th12} implies \eqref{th22}.
\end{remark}

\bigskip
The remaining part of this paper is organized as follows. Section \ref{section2}
contains some preparations needed for the proof of Theorem \ref{th1} and \ref{th2},
in particular a one-to-one relation between Nash-type
inequalities and estimates for the decay of the semigroups.
These estimates are needed for the proof of Theorem \ref{th1}
and \ref{th2} in Section \ref{section3}. Section \ref{section4} contains several
applications of our main result, e.g.\ the super-Poincar\'e and weak Poincar\'e
inequality for subordinate semigroups and the hyper-, super- and ultracontractivity
of subordinate semigroups.

\section{Preliminaries}\label{section2}
In this section we collect a few auxiliary results for the proof of
Theorems \ref{th1} and \ref{th2}. \rene We begin with a differential
and integral inequality, which is a special case of \cite[Appendix
A, Lemma A.1, p.\ 193]{Chen}. Note that the right hand side of the
inequality \eqref{lemma11} below is negative. This is different from
the usual Gronwall-Bellman-Bihari inequality, see e.g.\
\cite[Section 3]{Bi}, but it is essential for our purposes. For the
sake of completeness, we include the short proof from \cite[Appendix
A, the comment before Remark A.3, p.\ 194]{Chen}.\normal

\begin{lemma}\label{lemma1}  Let $h:[0,\infty)\to (0,\infty)$ be a differentiable function. Suppose that there exists an increasing
function $\varphi:(0,\infty)\rightarrow(0,\infty)$ such that
\begin{equation}\label{lemma11}
    h'(t)\le -\varphi (h(t))\quad\text{for all\ \ } t\geq 0.
\end{equation}
Then, we have
$$
    h(t)\le G^{-1}\bigg(G(h(0))-t\bigg)\quad\text{for all\ \ } t\geq 0,
$$
where $G^{-1}$ is the (generalized) inverse of
$$
    G(t) =
    \begin{cases}
        \displaystyle\phantom{-}\int_1^t\frac{du}{\varphi(u)}, & \text{if\ \ } t\geq 1,\\[\bigskipamount]
        \displaystyle-\int_t^1\frac{du}{\varphi(u)}, & \text{if\ \ } t\leq 1.
    \end{cases}
$$
\end{lemma}

\begin{proof}
    Since $h'(t)\le -\varphi(h(t))<0$ for all $t\ge0$, $h(t)$ is strictly decreasing on $[0,\infty)$. With the convention $\int_a^b = - \int_b^a$, we see for all $t\geq 0$ that
    \begin{align*}
    G(h(t))
    &= \int_1^{h(t)}\frac{1}{\varphi (u)}\,du\\
    &= G(h(0))+\int_{h(0)}^{h(t)}\frac{1}{\varphi (u)}\,du\\
    &= G(h(0))+\int_0^t \frac{h'(u)\,du}{\varphi(h(u))}\\
    &\le G(h(0))-t,
    \end{align*}
    and the claim follows.
\end{proof}

Let $(T_t)_{t\ge0}$ be a strongly continuous contraction semigroup of (not necessarily symmetric) operators on $L^2 = L^2(X,m)$. Denote by $(A,D(A)$ the infinitesimal generator. Since $\frac d{dt} T_tu = -T_t Au$ for all $u\in D(A)$, we have
$$
    \frac{d}{dt}\,\|T_tu\|_2^2
    = - 2 \Re \langle A \,T_tu,T_tu\rangle,\quad u\in D(A).
$$

\begin{proposition}\label{p1}
    Let $(T_t)_{t\geq 0}$ be a $C_0$ contraction semigroup on $L^2(X,m)$ and assume that each $T_t|_{L^2(X,m)\cap L^1(X,m)}$, $t\geq 0$, has an extension which is a contraction on $L^1(X,m)$, i.e.\ $\|T_tu\|_1\le \|u\|_1$ for all $u\in L^1(X,m)\cap L^2(X,m)$. Then the following Nash-type inequality
    \begin{equation}\label{th1122}
        \|u\|_2^2\,B\left( \|u\|_2^2\right)
        \le \Re\langle A\,u,u\rangle,
        \quad u\in D(A),\; \|u\|_1=1
    \end{equation}
    with some increasing function $B:(0,\infty)\rightarrow(0,\infty)$ holds if, and only if,
    \begin{equation}\label{pp1}
        \|T_t u\|_2^2\le G^{-1}\left(G(\|u\|_2^2)-t\right)
        \quad\text{for all $t\ge0$ and $u\in D(A), \; \|u\|_1=1$}
    \end{equation}
    where
    $$
    G(t) =
    \begin{cases}
        \displaystyle\phantom{-}\int_1^t\frac{ds}{2sB(s)}, & \text{if\ \ } t\geq 1,\\[\bigskipamount]
        \displaystyle-\int_t^1\frac{ds}{2sB(s)}, & \text{if\ \ } t\leq 1.
    \end{cases}
    $$
\end{proposition}

\begin{proof}
    Assume that \eqref{th1122} holds. Then,
    \begin{equation*}\label{th13}
        \|u\|_2^2\,B\left( \frac{\|u\|_2^2}{\|u\|_1^2}\right)
        \le \Re\langle A\,u,u\rangle,\quad u\in D(A).
    \end{equation*}
    For all $u\in D(A)$ with $\|u\|_1=1$ we have
    $$
        \frac{d}{dt}\,\|T_tu\|_2^2
        =-2\Re\langle A \,T_tu,T_tu\rangle
        \le
        -2\,\|T_tu\|_2^2\,B\left(\frac{\|T_tu\|_2^2}{\|T_tu\|_1^2}\right),
    $$
    Since the function $B$ is increasing and $\|T_tu\|_1\le \|u\|_1=1$, we have
    $$
        \frac{d}{dt}\,\|T_t u\|_2^2
        \le  -2\,\|T_tu\|_2^2\,B\left( {\|T_tu\|_2^2}\right).
    $$
    This, together with Lemma \ref{lemma1}, proves \eqref{pp1}.

    For the converse we assume that \eqref{pp1} holds. Then, for all $u\in D(A)$ with $\|u\|_1=1$,
    \begin{align*}
        \Re\langle Au,u\rangle
        &=-\frac{1}{2}\,\frac{d}{dt}\,\|T_t u\|_2^2\,\bigg|_{t=0}\\
        &=\frac{1}{2}\,\lim_{t\to0}\frac{\|u\|_2^2-\|T_t u\|_2^2}{t}\\
        &\ge\frac{1}{2}\,\lim_{t\to0}\frac{\|u\|_2^2-G^{-1}\left(G(\|u\|_2^2)-t\right)}{t}\\
        &=-\frac{1}{2}\,\frac{d}{dt}\,G^{-1}\left(G(\|u\|_2^2)-t\right)\,\bigg|_{t=0}\\
        &=\left[G^{-1}\left(G(\|u\|_2^2)-t\right)\cdot B\Big(G^{-1}\big(G(\|u\|_2^2)-t\big)\Big)\right]\,\bigg|_{t=0}\\
        &=\|u\|_2^2\,B(\|u\|_2^2),
    \end{align*}
    which is just the Nash-type inequality \eqref{th1122}.
\end{proof}

Finally we need some elementary estimate for Bernstein functions.
\begin{lemma}\label{l-okura}
    Let $f$ be a Bernstein function given by \eqref{bfunction} where $a=b=0$ and with representing measure $\nu$. Set
    $$
        \nu_1(x) := \int_0^x \nu\big(s,\infty\big)\,ds.
    $$
    Then for $x>0$,
    $$
        \frac{e-1}{e}\, x\,\nu_1\Big(\frac 1x\Big) \leq f(x) \leq x\,\nu_1\Big(\frac 1x\Big).
    $$
\end{lemma}

\begin{proof}
    By Fubini's theorem we find
    \begin{align*}
        x\,\nu_1\Big(\frac 1x\Big)
        = x\int_0^{1/x} \nu(s,\infty)\,ds
        &= \int_0^1 \nu\Big( \frac tx,\infty\Big)\,dt\\
        &= \int_0^1 \int_{t/x}^\infty \nu(dy)\,dt\\
        &= \int_0^\infty (xy\wedge 1)\,\nu(dy),
    \end{align*}
    see also \^Okura \cite[(1.5)]{OK}. Using the following elementary inequalities
    $$
        \frac{e-1}{e}(1\wedge r)
        \le 1-e^{-r}\le 1\wedge r
        \quad\text{for\ \ } r\ge0,
    $$
    we conclude
    $$
        \frac{e-1}e\, x\,\nu_1\Big(\frac 1x\Big)
        = \int_0^\infty\frac{e-1}e \, (xy\wedge 1)\,\nu(dy)
        \leq \int_0^\infty (1-e^{-xy})\,\nu(dy)
        = f(x).
    $$
    The upper bound follows similarly.
\end{proof}

\section{Proof of the main theorems}\label{section3}

\begin{proof}[\upshape\bfseries Proof of Theorem \ref{th1}]
Since $D(A)$ is an operator core for $(f(A),D(f(A))$, it is enough to prove \eqref{th12} for $u\in D(A)$. Using Phillip's formula \eqref{e-phillips} we find for all $u\in D(A)$
\begin{align*}
    \langle f(A)\, u,u\rangle
    &=a\,\|u\|_2^2 + b\,\langle A\,u, u\rangle+\int_{(0,\infty)} \langle u-T_su,u\rangle\,\nu(ds).
\end{align*}
This formula and the representation \eqref{bfunction} for $f$ show that we may, without loss of generality, assume that $a=b=0$.

Assume that \eqref{th11} holds.  Proposition \ref{p1} shows for $t\ge0$ and $u\in D(A)$ with $\|u\|_1=1$,
    $$
        \frac{\langle T_tu,u\rangle}{\|u\|_2^2}
        = \frac{\|T_{t/2}u\|_2^2}{\|u\|_2^2}
        \le \frac{G^{-1}\big(G(\|u\|_2^2)- t/2\big)}{\|u\|_2^2}.
    $$
    Then,
    \begin{align*}
    \langle f(A)\,u,u\rangle
    &= \int_{(0,\infty)} \langle u-T_su,u\rangle\, \nu(ds)\\
    &= \|u\|_2^2\, \int_{(0,\infty)} \left(1-\frac{\langle T_su,u\rangle}{\|u\|_2^2}\right)\nu(ds)\\
    &\geq \int_{(0,\infty)}\Big(\|u\|_2^2-{{G^{-1}\Big(G(\|u\|_2^2)- \frac s2\Big)}}\Big)\,\nu(ds)\\
    &= g(\|u\|_2^2),
    \end{align*}
    where
    \begin{align*}
    g(r)
    &=\int_{(0,\infty)} \Big(r-{G^{-1}\Big(G(r)- \frac s2\Big)}\Big)\,\nu(ds).
    \end{align*}
    Furthermore, for all $r>0$,
    \begin{align*}
    g(r)
    &=  \int_{(0,\infty)} \Big(r-{G^{-1}\Big(G(r)-\frac s2\Big)}\Big)\,\nu(ds)\\
    &= \int_{(0,\infty)} \left(\int_{G(r)-s/2}^{G(r)}dG^{-1}(u)\right)\nu(ds)\\
    &= \int_{-\infty}^{G(r)} \nu\Big(2\big(G(r)-u,\infty\big)\Big)\,dG^{-1}(u)\\
    &= \int_0^r \nu\Big(2\big(G(r)-G(u)\big),\infty\Big)\,du.
    \end{align*}
    For the last equality we used that $B$ is increasing, $G(x)>-\infty$ for all $x>0$ and $G(0)=-\infty$; this follows from
    $$
        G(0)
        = -\int_0^1\frac{du}{u B(u)}
        \le \frac{-1}{B(1)}\int_0^1\frac{du}{u}
        =-\infty.
    $$

    Using again the monotonicity of $B$, we find from the mean value theorem
    \begin{equation}\label{e-convex}
        \frac{1}{2uB(u)}\ge \frac{G(r)-G(u)}{r-u}\ge \frac{1}{2rB(r)}
        \quad
        \text{for all}
        \quad
        0<u<r.
    \end{equation}
    Therefore,
    \begin{align}
    g(r)
    &\ge\notag \int_0^r \nu\left(\frac{1}{uB(u)}(r-u),\infty\right) du\\
    &\ge\label{e-epsilon} \int_{r/2}^r \nu\left(\frac{1}{uB(u)}(r-u),\infty\right)du\\
    &\ge\notag \int_0^{r/2} \nu\left(\frac{2v}{rB(r/2)},\infty\right)dv\\
    &=\notag \frac 12\, rB(r/2) \int_0^{1/B(r/2)} \nu\big(s,\infty\big)\,ds.
    \end{align}

    A similar calculation, now using the lower bound in \eqref{e-convex}, yields
    $$
        g(r)
        \leq rB(r) \int_0^{1/B(r)} \nu\big(s,\infty\big)\,ds.
    $$
    Now we can use Lemma \ref{l-okura} to deduce that
    $$
        \frac{e}{e-1}\,rf(B(r))\ge g(r)\ge \frac r2\,f\left(B\left(\frac r2\right)\right)
        \quad\text{for all}\quad r>0,
    $$
    and the proof is complete. 
    \end{proof}
\begin{remark}\label{remark}
(i) In the proof of Theorem \ref{th1}, at the line \eqref{e-epsilon}, we can replace  $r/2$ by $\varepsilon r$ for $\varepsilon \in (0,1)$. Then we get
\begin{align*}
    g(r)
    &\ge \sup_{\varepsilon\in (0,1)}\bigg[(1-\varepsilon)\,r\,f\bigg( \frac{\varepsilon B(\varepsilon r)}{1-\varepsilon}\bigg)\bigg],
\end{align*}
    which shows that we can improve \eqref{th12} by
    $$
        \sup_{\varepsilon\in (0,1)}\bigg[(1-\varepsilon)\,\|u\|_2^2\,f\bigg( \frac{\varepsilon B(\varepsilon \|u\|_2^2)}{1-\varepsilon}\bigg)\bigg]
        \le \langle f(A)\,u,u\rangle,
        \quad u\in D(f(A)),\; \|u\|_1=1.
    $$

\medskip\noindent
(ii) A close inspection of our proof shows that Theorem \ref{th1} remains valid if we replace the norming
condition $\|u\|_1=1$ in \eqref{th11} and \eqref{th21} by the more general condition $\Phi(u)=1$. Here
$\Phi : L^2(X,m) \to [0,\infty]$ is a measurable functional satisfying $\Phi(cu)=c^2\Phi(u)$ and $\Phi(T_tu)\le \Phi(u)$ for all $t\geq 0$.
\end{remark}

\begin{proof}[\upshape\bfseries Proof of Theorem \ref{th2}]
The proof of Theorem \ref{th2} is similar to the proof of Theorem \ref{th1}. Therefore we only outline the differences in the arguments. As before we can assume that $f(\lambda)=\int_{(0,\infty)}\big(1-e^{-t\lambda}\big)\,\nu(dt)$. Moreover, it is enough to verify \eqref{th22} for all $u\in D(A)$. Since $(T_t)_{t\ge0}$ is a contraction on $L^1(X,m)\cap L^2(X,m)$, we see from \eqref{th21} and Proposition \ref{p1} that for all $t\ge0$ and $u\in D(A)$ with $\|u\|_1=1$,
$$
    \|T_t u\|_2^2\le G^{-1}\left(G(\|u\|_2^2)-t\right).
$$
By the Cauchy-Schwarz inequality,
$$
    \frac{\Re \langle T_tu,u\rangle}{\|u\|_2^2}
    \le \frac{|\langle T_tu,u\rangle|}{\|u\|_2^2}
    \le \frac{\|T_tu\|_2\|u\|_2}{\|u\|_2^2}
    \le \frac{\sqrt{G^{-1}\Big(G(\|u\|_2^2)-t\Big)}}{\|u\|_2}.
$$
Using \eqref{e-phillips} yields that for any $u\in D(A)$ with $\|u\|_1=1$,
\begin{align*}
    \Re\langle f(A)\,u,u\rangle
    &= \int_{(0,\infty)} \Re \langle u-T_su,u\rangle\, \nu(ds)\\
    &= \|u\|_2^2\, \int_{(0,\infty)} \left(1-\frac{\Re \langle T_su,u\rangle}{\|u\|_2^2}\right)\nu(ds)\\
    &\ge \|u\|_2^2\,\int_{(0,\infty)} \left(1-\dfrac{\sqrt{G^{-1}(G(\|u\|_2^2)-s)}}{\|u\|_2}\right)\nu(ds)\\
    &= \|u\|_2^2\,\int_{(0,\infty)}
        \left(\dfrac{1-\dfrac{{G^{-1}(G(\|u\|_2^2)-s)}}{\|u\|^2_2}}%
                    {1+\dfrac{\sqrt{G^{-1}(G(\|u\|_2^2)-s)}}{\|u\|_2}}
        \right) \nu(ds)\\
    &\ge \frac{\|u\|_2^2}{2} \int_{(0,\infty)}\left(1-\dfrac{{G^{-1}\big(G(\|u\|_2^2)-s\big)}}{\|u\|^2_2}\right) \nu(ds)\\
    &= g(\|u\|_2^2),
\end{align*} where $$g(r)
    = \frac{r}{2}\int_{(0,\infty)} \left(1-\frac{{G^{-1}(G(r)-s)}}{r}\right) \nu(ds).$$
A similar calculation as in the proof of Theorem \ref{th1} shows
\begin{align*}
    g(r)
    = \frac{1}{2} \int_0^r \nu\big(G(r)-G(u),\infty\big)\,du
    \ge\frac{r}{4}\,f\Big(2\,B\Big(\frac{r}{2}\Big)\Big),
\end{align*}
which is exactly \eqref{th22}.
\end{proof}

\section{Applications}\label{section4}
We will now give some applications of our results. Throughout this section we retain the notations introduced in the previous sections. In particular,  $(T_t)_{t\geq 0}$ will be a strongly continuous contraction semigroup on $L^2(X,m)$ with generator $(A,D(A))$. We assume that $\|T_t u\|_1 \leq \|u\|_1$ for all $u\in L^2(X,m)\cap L^1(X,m)$ and, for simplicity, that the operators $T_t$, $t\geq 0$, are symmetric. By $\Phi: L^2(X,m)\rightarrow [0,\infty]$ we denote a functional on $L^2(X,m)$ such that for all $c,t>0$ and $u\in L^2(X,m)$
$$
        \Phi(cu)=c^2\Phi(u)
        \quad\text{and}\quad
        \Phi\,(T_tu)\le \Phi(u);
$$
by $f$ we always denote a Bernstein function given by \eqref{bfunction}.

\subsection{Subordinate super-Poincar\'e inequalities}\label{subsec-super-poincare}
In this section, we study the analogue of Theorem \ref{th1} for super-Poincar\'e inequalities. For details on super-Poincar\'e inequalities and their applications we refer to \cite{wang1, wang11, wang2} or \cite[Chapter 3]{Wang}.
\begin{proposition}\label{super} Assume that $(A,D(A))$ satisfies the following super-Poincar\'e inequality:
    \begin{equation}\label{th31}
        \|u\|_2^2\le r\, \langle A\,u,u\rangle+ \beta(r)\,\Phi (u),
        \quad r>0,\; u\in D(A),
    \end{equation}
    where $\beta:(0,\infty)\rightarrow(0,\infty)$ is a decreasing function such that $\lim_{r\to 0}\beta(r)=\infty$ and $\lim_{r\to\infty}\beta(r)=0$; moreover, we set $\beta(0):=\infty$. Then the generator $f(A)$ of the subordinate semigroup also satisfies a super-Poincar\'e inequality
    \begin{equation}\label{th32}
    \|u\|_2^2\le r\, \langle f(A)\,u,u\rangle+ {\beta}_f(r)\,\Phi (u),
        \quad r>0,\; u\in D(f(A)),
    \end{equation}
    where
    $$
        {\beta}_f(r)= 4\beta\bigg(\frac{1}{2f^{-1}(2/r)}\bigg)
    $$
\end{proposition}
    \begin{proof}
        We can rewrite \eqref{th31} for any $u\in D(A)$ with $\Phi(u)=1$ in the following form:
        $$
            \|u\|_2^2\,B(\|u\|_2^2)\le \langle A\,u,u\rangle,
        $$
        where
        $$
            B(x)=\sup_{s>0}\frac{1-\beta(s)/x}{s}.
        $$
        Clearly, $B(x)$ is an increasing function on $(0,\infty)$. Since $\beta^{-1}:(0,\infty)\rightarrow(0,\infty)$, we see from
       \begin{equation}\label{proofth31}
            \frac{1}{2\beta^{-1}\big(x/2\big)}
            =\frac{1-\beta\big(\beta^{-1}(x/2)\big)/x}{\beta^{-1}(x/2)}\le B(x)
            =\sup_{s\ge \beta^{-1}(x)}\frac{1-\beta(s)/x}{s}
            \le\frac{1}{\beta^{-1}(x)}
        \end{equation}
        that $B:(0,\infty)\rightarrow(0,\infty)$.

        Using Theorem \ref{th1} and the Remark \ref{remark} (ii) yields for any $u\in D(f(A))$ with $\Phi(u)=1$,
        $$\Theta(\|u\|^2_2)\le\langle f(A)\,u,u\rangle,$$ where $$\Theta(x)=\frac{x}{2}\,f\Big(B\Big(\frac{x}{2}\Big)\Big)=\frac{x}{2}\,\sup_{s>0}\,f\bigg(\frac{1-2\beta(s)/x}{s}\bigg).$$ For $r>0$, define
$$
    \widetilde{\beta}(r)
    =\sup_{s>0}\Big\{\Theta^{-1}(s)-rs\Big\}.
$$
Then,
\begin{equation}\label{proofth32}
    \|u\|_2^2
    \le r\, \langle f(A)\,u,u\rangle + \widetilde{\beta}(r)\,\Phi (u),
        \quad r>0,\; u\in D(f(A)).
\end{equation}

Next, we will estimate $\widetilde{\beta}(r)$. By \eqref{proofth31},
$$
    \Theta(x)\ge\frac{x}{2}\,f\bigg(\frac{1}{2\,\beta^{-1}(x/4)}\bigg):=\Theta_0(x),
$$
which in turn implies that
$$
    \Theta^{-1}(x)\le \Theta_0^{-1}(x).
$$
By the definition of $\Theta_0(x)$, $\Theta_0:(0,\infty)\rightarrow(0,\infty)$ is a strictly increasing function such that $\lim_{x\to 0}\Theta_0(x)=0$ and $\lim_{x\to\infty}\Theta_0(x)=\infty$, and so
\begin{equation}\label{proofth33}
    \Theta_0^{-1}(x)
    = {2x}\left[{f\left(\frac{1}{2\,\beta^{-1}(\Theta_0^{-1}(x)/4)}\right)}\right]^{-1}.
\end{equation}
On the other hand,
\begin{equation*}\label{proofth34}
   \widetilde{\beta}(r)
   \le\sup_{s>0}\Big\{\Theta_0^{-1}(s)-rs\Big\}
   =\sup_{s>0,\,\Theta_0^{-1}(s)\ge rs}\Theta_0^{-1}(s).
\end{equation*}
From \eqref{proofth33} we see that $\Theta_0^{-1}(s)\ge rs$ is equivalent to
$$
    \frac{1}{2 f^{-1}(2/r)}
    \le \beta^{-1}\left(\frac{\Theta_0^{-1}(s)}{4}\right).
$$
Since $\beta$ is decreasing, we can rewrite this as
$$
    \Theta_0^{-1}(s)
    \le 4\beta\left(\frac{1}{2\,f^{-1}(2/r)}\right),
$$
and so
\begin{equation}\label{proofth34}
   \widetilde{\beta}(r)
   \le\sup_{s>0,\,\Theta_0^{-1}(s)\le 4\beta\big(\frac{1}{2\,f^{-1}(2/r)}\big)} \Theta_0^{-1}(s)
   \le 4\beta\left(\frac{1}{2\,f^{-1}(2/r)}\right).
\end{equation}
The proof is complete if we combine \eqref{proofth32} and \eqref{proofth34}.
\end{proof}

\subsection{Subordinate weak Poincar\'e inequalities}\label{subsec-weak-poincare}
We can also consider the subordination for weak Poincar\'{e} inequalities; for details we refer to \cite{RW} or
\cite[Chapter 4]{Wang}.
\begin{proposition}\label{weak}
    Assume that $(A,D(A))$ satisfies the following weak Poincar\'e inequality:
    \begin{equation}\label{th312}
        \|u\|_2^2\le \alpha(r)\, \langle A\,u,u\rangle+ r\,\Phi (u),
        \quad r>0,\; u\in D(A),
    \end{equation}
    where $\alpha:(0,\infty)\rightarrow(0,\infty)$ is a decreasing function. Then the generator $f(A)$ of the subordinate semigroup also satisfies a weak Poincar\'e inequality
   \begin{equation}\label{th322}
    \|u\|_2^2\le {\alpha}_f(r)\, \langle f(A)\,u,u\rangle+ r\,\Phi (u),
        \quad r>0,\; u\in D(f(A)),
    \end{equation}
    where
    $$
        {\alpha}_f(r)
        ={2}\bigg/\left[f\left(\frac{1}{2\,\alpha(r/4)}\right)\right].
    $$
\end{proposition}
\begin{proof}
Suppose that \eqref{th312} holds. As in the proof of Proposition \ref{super} we find that
\begin{equation}\label{proofth36}
    \|u\|_2^2\le \widetilde{\alpha}(r) \langle f(A)\,u,u\rangle+ r\,\Phi (u),
    \quad r>0,\; u\in D(f(A)),
\end{equation}
where
$$
    \widetilde{\alpha}(r)=\sup_{s>0}\bigg\{\frac{\Theta^{-1}(s)-r}{s}\bigg\}
\quad
\text{and}
\quad
    \Theta(x)=\frac{x}{2}\,\sup_{s>0}\,f\bigg(\frac{1-2s/x}{\alpha(s)}\bigg).
$$
If we set $s={x}/{4}$,
$$
    \Theta(x)
    \ge \frac{x}{2}\,f\bigg(\frac{1}{2\,\alpha(x/4)}\bigg):=\Theta_0(x),
$$
and this gives us
\begin{equation}\label{proofth37}
    \widetilde{\alpha}(r)
    = \sup_{s>0}\left\{\frac{\Theta_0^{-1}(s)-r}{s}\right\}
    \le \sup_{s>0,\,\Theta_0^{-1}(s)\ge r} \frac{\Theta_0^{-1}(s)}{s}.
\end{equation}
According to the definition of $\Theta_0(x)$, we have
$$
    \frac{\Theta_0^{-1}(x)}{x}={2}\left[{f\left(\frac{1}{2\,\alpha\big(\Theta_0^{-1}(x)/4\big)}\right)}\right]^{-1}.
$$
Since $\alpha$ is decreasing,
\begin{equation}\label{proofth38}\begin{aligned}
    \sup_{s>0,\,\Theta_0^{-1}(s)\ge r }\frac{\Theta_0^{-1}(s)}{s}
    &= \sup_{s>0,\,\Theta_0^{-1}(s)\ge r}{2}\bigg[{f\bigg(\frac{1}{2\,\alpha(\Theta_0^{-1}(s)/4)}\bigg)}\bigg]^{-1}\\
    &\le {2}\bigg[{f\bigg(\frac{1}{2\,\alpha(r/4)}\bigg)}\bigg]^{-1}.
\end{aligned}\end{equation}
The required inequality \eqref{th322} follows from \eqref{proofth36}, \eqref{proofth37} and
\eqref{proofth38}.
\end{proof}

\subsection{The converses of Theorem \ref{th1} and Propositions \ref{super} and \ref{weak}}

If $A$ is a nonnegative self-adjoint operator, then it is possible to show a converse to the assertions of Theorem \ref{th1} and Propositions \ref{super} and \ref{weak}.
\begin{proposition}\label{probb}
Let $A$ be a nonnegative self-adjoint operator on $L^2(X,m)$, and $f$ be some non-degenerate Bernstein function. Let $\Phi : L^2(X,m) \to [0,\infty]$ be a measurable
functional satisfying $\Phi(cu)=c^2\Phi(u)$ and $\Phi(T_tu)\le
\Phi(u)$ for all $t\geq 0$, where $(T_t)_{t\geq 0}$ is the semigroup
generated by $A$.
    If the following Nash-type inequality
    $$
        {\|u\|_2^2}\,f\left(B\left({\|u\|_2^2}\right)\right)
        \le \langle f(A)\,u,u\rangle,
        \quad u\in D(f(A)),\; \Phi(u)=1
    $$
    holds for some increasing function $B:(0,\infty)\rightarrow(0,\infty)$, then
    $$
        {\|u\|_2^2}\,B\left({\|u\|_2^2}\right)
        \le \langle A\,u,u\rangle,
        \quad u\in D(A),\; \Phi(u)=1.
    $$
\end{proposition}

\begin{proof}
    Every non-degenerate (i.e.\ non-constant) Bernstein function $f$ is strictly increasing and concave. Thus $f^{-1}$ is strictly increasing and convex. Let $(E_\lambda)_{\lambda\geq 0}$ be the spectral resolution of the self-adjoint operator $A$. Using Jensen's inequality we get for all $u\in D(A)$ with $\nnorm u_1=1$
    \begin{align*}
        B\left(\nnorm u_2^2\right)
        &= f^{-1}\circ f\left(B\left(\nnorm u_2^2\right)\right)\\
        &\leq f^{-1}\left(\frac{\scalar{f(A)u,u}}{\nnorm u_2^2}\right)\\
        &= f^{-1}\left(\int_{[0,\infty)} f(\lambda) \,\frac{dE_\lambda(u,u)}{\nnorm u_2^2}\right)\\
        &\le \int_{[0,\infty)} f^{-1}\circ f(\lambda) \,\frac{dE_\lambda(u,u)}{\nnorm u_2^2}\\
        &= \frac{\scalar{Au,u}}{\nnorm u_2^2},
    \end{align*}
    cf.\ also \cite[Proposition 2.3]{BM}.
\end{proof}

Using Proposition \ref{probb} we can get the converses of Propositions \ref{super} and \ref{weak}. For example, if the following super-Poincar\'{e} inequality
    $$
        \|u\|_2^2
        \le r\, \langle f(A)\,u,u\rangle+ {\beta}_f(r)\,\Phi (u),
        \quad r>0,\; u\in D(f(A))
    $$
    holds for some decreasing function $\beta_f:(0,\infty)\rightarrow(0,\infty)$, then
    $$
        \|u\|_2^2
        \le r\, \langle A\,u,u\rangle+ \beta(r)\,\Phi (u),
        \quad r>0,\; u\in D(A),
    $$
    where
    $$
        \beta(r)=2\,\beta_f\bigg(\frac{1}{2\,f(1/r)}\bigg).
    $$

\subsection{On-diagonal estimates for subordinate semigroups:
Nash type inequalities} In this section $X$ is the $n$-dimensional
Euclidean space $\R^n$ equipped with Lebesgue measure $m(dx)=dx$.
\begin{proposition}
Assume that $(A,D(A))$ satisfies the following Nash-type inequality
\begin{equation*}\label{e-rn-nash}
    \|u\|_2^2\, B\big(\|u\|_2^2\big) \leq \langle Au,u\rangle,
    \quad
    u\in D(A),\; \|u\|_1=1,
\end{equation*}
where $B:(0,\infty)\to (0,\infty)$ is some increasing function. Then, if for any $t>0$,
$$
    \eta(t):=\int_t^\infty \frac{du}{u\,f(B(u))}<\infty,
$$
the subordinate semigroup $(T_t^f)_{t\geq 0}$ has a bounded kernel $p^f_t(x,y)$ with respect to Lebesgue measure, and the following on-diagonal estimate holds:
\begin{equation*}\label{ultra-contractive}
   \mathop{\mathrm{ess\,sup}}_{x,y\in\R^d}p^f_t(x,y)=\big\|T_t^f\big\|_{1\to\infty}
    \leq 2\,\eta^{-1}\Big(\frac{t}{2}\Big).
\end{equation*}
\end{proposition}

\begin{proof}
    By Theorem \ref{th1} we know that the generator $f(A)$ of the subordinate semigroup $(T_t^f)_{t\geq 0}$ satisfies
\begin{equation}\label{e-rn-nash2}
   \frac{ \|u\|_2^2}{2} \, f\circ B\left(\frac{\|u\|_2^2}{2}\right) \leq \langle f(A)\,u,u\rangle,
    \quad
    u\in D(f(A)),\; \|u\|_1=1.
\end{equation}
     Therefore the required assertion follows from \cite[Proposition II.2]{Con} or \cite[Theorem 3.3.17 (1), p.\ 158]{Wang}.
\end{proof}

\subsection{Contractivity of subordinate semigroups: Super- and Weak Poincar\'{e} inequalities}
  Let $(X,m)$ be a measure space with a $\sigma$-finite measure $m$.
    Let $(T_t)_{t\geq 0}$ be a semigroup on
    $L^2(X,m)$ which is bounded on $L^p(X,m)$ for all
    $p\in[1,\infty]$. This is, e.g., always the case for symmetric sub-Markovian contraction semigroups on $L^2(X,m)$.

    Recall that a semigroup $(T_t)_{t\ge0}$ is said to be \emph{hypercontractive} if $\|T_t\|_{2\to4}<\infty$ for some $t>0$, \emph{supercontractive} if $\|T_t\|_{2\to4}<\infty$ for all
    $t>0$, and \emph{ultracontractive} if $\|T_t\|_{1\to\infty}<\infty$ for all $t>0$.
    The example below improves \cite[Theorem 3.1]{BM}.

\begin{proposition}
    Let $f$ be a Bernstein function and $(T_t)_{t\geq 0}$ be an ultracontractive symmetric sub-Markovian semigroup on $L^2(X,m)$ such that for all $t>0$,
    $$
        \|T_t\|_{1\to\infty}
        \le \exp\Big(\lambda\,t^{-1/(\delta-1)}\Big)
    $$
    for some $\lambda>0$ and $\delta>1$. Then, we have the following statements for the subordinate semigroup $(T_t^f)_{t\ge0}$:
    \begin{enumerate}[\upshape (i)]
    \item If
        $\displaystyle \int_1^\infty \frac{dr}{f(r^\delta)}<\infty$, then $(T_t^f)_{t\ge0}$ is ultracontractive.
    \item
        If $\displaystyle \lim_{r\rightarrow\infty}\frac{f^{-1}(\lambda)}{\lambda^\delta}=0$, then $(T_t^f)_{t\ge0}$ is supercontractive.
    \item
        If $\displaystyle \lim_{r\rightarrow\infty}\frac{f^{-1}(\lambda)}{\lambda^\delta}\in(0,\infty)$, then $(T_t^f)_{t\ge0}$ is hypercontractive.
    \item
        If $\displaystyle \lim_{r\rightarrow\infty}\frac{f^{-1}(\lambda)}{\lambda^\delta}=\infty$, then $(T_t^f)_{t\ge0}$ is not hypercontractive.
    \end{enumerate}
\end{proposition}

\begin{proof}
    Denote by $A$ and $f(A)$ the generators of the semigroups $(T_t)_{t\geq 0}$ and $(T^f_t)_{t\geq 0}$, respectively. By \cite[Proposition II. 4]{Con} and \cite[Proposition 3.3.16, p.\ 157]{Wang}, we know that the following super-Poincar\'{e} inequality holds:
    $$
        \|u\|_2^2
        \le r\, \langle A\,u,u\rangle+ \beta(r)\,\|u\|_1^2,
        \quad r>0,\; u\in D(A),
    $$
    where
        $$
        \beta(r)=c_1\Big[\exp\big(c_2\,r^{-1/\delta}\big)-1\Big]
    $$
    for some $c_1,c_2>0$. By Proposition \ref{super},
    $$
        \|u\|_2^2
        \le r\, \langle f(A)\,u,u\rangle+ \beta_f(r)\,\|u\|_1^2,
        \quad r>0,\; u\in D(f(A)),
    $$
    where
    $$
        \beta_f(r)
        =4c_1\bigg\{\exp\left[c_3\left(f^{-1}\big(2/r\big)\right)^{1/\delta}\right]-1\bigg\}
    $$
    for some constant $c_3>0$. Therefore, the required assertions follow from
    \cite[Theorem 3.3.14, p.\ 156 and Theorem 3.3.13, p.\ 155]{Wang} and the comment after Proposition \ref{probb}.
\end{proof}

We close this section with a result that shows how decay properties are inherited under subordination.
\begin{proposition}
    Let $(T_t)_{t\geq 0}$ be a symmetric sub-Markovian semigroup on $L^2(X,m)$. Assume that there exist two constants $\delta, c_0>0$ such that
    $$
        \|T_tu\|^2_2
        \le \frac{c_0\, \Phi(u)}{t^\delta}
        \quad\text{for all\ \ } t>0,\; u\in L^2(X,m),
    $$
    where $\Phi: L^2(X,m)\rightarrow [0,\infty]$ is a functional satisfying
    $\Phi(cu)=c^2\Phi(u)$ and $\Phi(T_tu)\le \Phi(u)$ for all $c\in \R$ and $t\ge0$. If
    $$
        \eta(t):=\int_t^\infty\frac{ds}{sf(s)}
        < \infty
        \quad\text{for all\ \ } t>0,
    $$
    then there are constants $c_1,c_2>0$ such that
    $$
        \|T_t^fu\|_2^2
        \le c_1\,\Big[\eta^{-1}(c_2{t})\Big]^\delta \,\Phi(u).
    $$
\end{proposition}
\begin{proof}
    Denote by $A$ and $f(A)$ the generators of $(T_t)_{t\geq 0}$ and $(T^f_t)_{t\geq 0}$, respectively.
    From \cite[Corollary 4.1.8 (1), p.\ 189; and Corollary 4.1.5 (2), p.\ 186]{Wang} we know that the following weak Poincar\'{e} inequality holds:
    $$
        \|u\|_2^2
        \le \alpha(r)\, \langle A\,u,u\rangle+ r\,\Phi (u),
        \quad r>0,\; u\in D(A),
    $$
    where
    $$
        \alpha(r)=c_3\,r^{-1/\delta}
    $$
    for some $c_3>0$. Proposition \ref{weak} shows that
    $$
        \|u\|_2^2
        \le \alpha_f(r)\, \langle f(A)\,u,u\rangle+ r\,\Phi (u),
        \quad r>0,\; u\in D(f(A)),
    $$
    where
    $$
        \alpha_f(r)={2}\Big[{f\big(c_4 r^{1/\delta}\big)}\Big]^{-1}
    $$
    for some constant $c_4>0$. Therefore, the assertion follows from \cite[Theorem 4.1.7, p.\ 188]{Wang}.
\end{proof}

\bigskip

\begin{ack}
The authors would like to thank
Professors Mu-Fa Chen and Feng-Yu Wang for helpful comments on earlier versions of the paper.
Financial support through DAAD (for R.L.S.) and
the Alexander-von-Humboldt Foundation
(for J.W.)
is gratefully
acknowledged.
\end{ack}

\end{document}